\documentclass[10pt]{article}

\usepackage{amsfonts, amsmath, paralist}
\usepackage{caption, graphicx}
\usepackage[text={6.5in,9.25in},centering]{geometry}
\usepackage[numbers,comma,square,sort&compress]{natbib}

\newtheorem{hypothesis}{Hypothesis}

\newtheorem{definition}{Definition}

\newtheorem{theorem}{Theorem}

\newtheorem{lemma}{Lemma}

\makeatletter\@addtoreset{equation}{section}\makeatother

\newenvironment{proof}[1][.]%
 {\begin{trivlist}\item[]\textbf{Proof#1 }}%
 {\hspace*{\fill}$\rule{0.3\baselineskip}{0.35\baselineskip}$\end{trivlist}}

\def\Fix{\mathop\mathrm{Fix}\nolimits}  % fixed point space
\newcommand{\N}{\mathbb{N}}             % natural numbers
\newcommand{\R}{\mathbb{R}}             % reals
\newcommand{\Z}{\mathbb{Z}}             % integers

\newcommand{\rmO}{\mathrm{O}}           % Landau O
\newcommand{\rmD}{\mathrm{D}}           % derivatives
\newcommand{\rme}{\mathrm{e}}           % Euler constant
\newcommand{\rmi}{\mathrm{i}}           % imaginary unit

\newcommand{\rpf}{\mathbf{u}} 
\newcommand{\wl}{k} 

\begin{document}

\title{Branches of localized patterned states}
\author{Bj\"orn Sandstede\\ Division of Applied Mathematics, Brown University}
\date{\today}
\maketitle

\begin{abstract}
Motivated by theoretical analyses of spatially localized structures with arbitrarily long periodic plateaus, we provide a framework of assumptions that simplifies their analysis and leads to a topological criterion for when localized patterned structures lie on a discrete stack of loops or on a single unbounded branch. The framework proposed here also connects closely with continuation algorithms that are often used to verify the hypotheses that guarantee the emergence of localized patterned states.
\end{abstract}

%%%%%%%%%%%%%%%%%%%%%%%%%%%%%%%%%%%%%%%%%%%%%%%%%%%%%%%%%%%%%%%%%%%%%%%%%%%%

\section{Introduction}\label{s:0}

\begin{figure}
\centering
\includegraphics{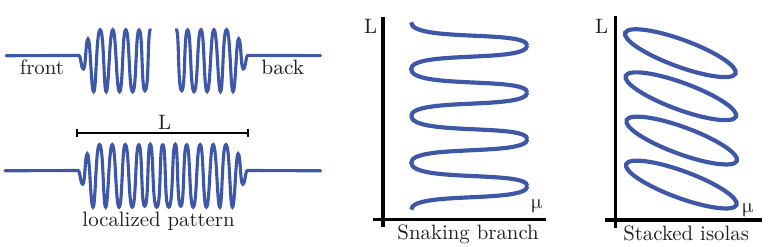}
\caption{The left panel illustrates the graphs of a patterned front $U_\mathrm{het}(x)$, its symmetric counterpart $U_\mathrm{het}(-x)$, which forms a patterned back, and a localized patterned state $U_L(x)$ as functions of $x$. The center and right panel contain an unbounded snaking branch and a stack of loops along which localized patterned states exist, where $L$ denotes the length of the patterned plateau and $\mu$ is a system parameter.}
\label{f:1}
\end{figure}

Spatially localized structures with long spatially periodic plateaus in their core have been observed in many spatially extended systems modeled by partial differential equations (PDEs), and we refer to the review articles \citep{Dawes, K1, K2, Lloyd} for an overview of old and new experimental, computational, and theoretical results that illustrate and explain the emergence of localized patterns. We focus on one-dimensional stationary localized patterns $U(x)$ with $x\in\R$ so that $U(x)\to0$ as $|x|\to\infty$ and so that there is an $L\gg1$ such that $U(x)$ resembles a spatially periodic function $U_\mathrm{wt}(x)$ for $|x|\leq L$; see Figure~\ref{f:1} for an illustration.

Consider a PDE that depends on a system parameter $\mu$ and assume that there is a closed loop of patterned fronts $U_\mathrm{het}(x,s)$ for $\mu=\mu(s)$ parametrized by $s\in S^1$ that connect $U=0$ to a spatially periodic profile $U_\mathrm{wt}(x,s)$ so that $U_\mathrm{het}(x,s)\to0$ as $x\to-\infty$ and $|U_\mathrm{het}(x,s)-U_\mathrm{wt}(x,s)|\to0$ as $x\to\infty$ for each $s$. In \citep{Snaking}, it was shown that this parametrized branch causes the emergence of spatially localized patterns $U_L(x)$ with arbitrarily long periodic plateaus of length $L$ that lie either on a discrete stack of closed branches (referred to as isolas) or on a single unbounded branch (referred to as snaking); see again Figure~\ref{f:1} for an illustration.

However, some of the hypotheses in \citep{Snaking} are formulated in a normal-form coordinate system near the periodic profile $U_\mathrm{wt}(x)$, and they rely on specific parametrizations of invariant manifolds in these coordinates that are difficult to verify numerically. In these notes, we provide a framework of hypotheses that implies the same results and involves only transversality criteria, which are easier to verify. As a byproduct, this framework also clarifies the distinction between stacked isolas and snaking through the homotopy class of an appropriate asymptotic phase function.

%%%%%%%%%%%%%%%%%%%%%%%%%%%%%%%%%%%%%%%%%%%%%%%%%%%%%%%%%%%%%%%%%%%%%%%%%%%%

\section{A continuation framework for localized patterned states}\label{s:1}

We consider ordinary differential equations
\begin{equation}\label{e:ode}
u_x = f(u,\mu), \qquad u\in\R^4, \quad \mu\in\R
\end{equation}
where $f$ is $C^\infty$. We assume that this system admits a reverser.

\begin{hypothesis}\label{h:1}
There exists a linear map $\mathcal{R}:\R^4\to\R^4$ with $\mathcal{R}^2=1$ and $\dim\Fix\mathcal{R}=2$ so that $f(\mathcal{R}u,\mu)=-\mathcal{R}f(u,\mu)$ for all $(u,\mu)$.
\end{hypothesis}

We say that a solution of (\ref{e:ode}) is symmetric if $u(0)\in\Fix\mathcal{R}$. Symmetric solutions $u(x)$ satisfy $u(-x)=\mathcal{R}u(x)$ for all $x\in\R$. Next, we assume that $u=0$ is a hyperbolic equilibrium for all $\mu$.

\begin{hypothesis}\label{h:2}
The origin $u=0$ is a hyperbolic equilibrium of (\ref{e:ode}) so that $f(0,\mu)=0$ for all $\mu$ and $f_u(0,\mu)$ has two eigenvalues with strictly negative real part and another two eigenvalues with strictly positive real part.
\end{hypothesis}

Throughout, we set $S^1:=\R/(2\pi\Z)$ and define $\mathcal{X}:=\R\times\R^+\times C^1(S^1,\R^4)\times C^1(\R,\R^4)$. An element $\rpf_*=(\mu_*,\wl_*,u_\mathrm{wt}^*,u_\mathrm{het}^*)\in\mathcal{X}$ is said to be a \emph{patterned front} if it satisfies the following conditions:

\begin{definition}[Patterned front]\label{d:pf}
An element $\rpf_*=(\mu_*,\wl_*,u_\mathrm{wt}^*,u_\mathrm{het}^*)\in\mathcal{X}$ is called a \emph{patterned front} if the following is true:
\begin{compactenum}
\item The functions $u(x)=u_\mathrm{wt}^*(\wl_*x)$ and $u_\mathrm{het}^*(x)$ are solutions of (\ref{e:ode}) for $\mu=\mu_*$.
\item The periodic orbit $u(x)=u_\mathrm{wt}^*(\wl_*x)$ of (\ref{e:ode}) with $\mu=\mu_*$ satisfies $u_\mathrm{wt}^*(0)\in\Fix\mathcal{R}$, and it has two Floquet multipliers at $\rme^{\pm2\pi\alpha_*}$ for some $\alpha_*>0$.
\item We have $u_\mathrm{het}^*(x)\in W^\mathrm{s}(0,\mu_*)$, and there is a $\psi_*\in S^1$ so that $|u_\mathrm{het}^*(x)-u_\mathrm{wt}^*(\wl_*x+\psi_*)|\to0$ as $x\to-\infty$.
\end{compactenum}
\end{definition}

Assume that $\rpf_*=(\mu_*,\wl_*,u_\mathrm{wt}^*,u_\mathrm{het}^*)\in\mathcal{X}$ is a patterned front with the associated periodic orbit $u_\mathrm{wt}^*(\wl_*\cdot)$. The extended system
\[
\begin{pmatrix} u_x \\ \mu_x \end{pmatrix} = \begin{pmatrix} f(u,\mu) \\ 0 \end{pmatrix}, \qquad (u,\mu)\in\R^4\times\R
\]
then admits the invariant smooth manifolds $\widetilde{W}^\mathrm{cu}(u_\mathrm{wt}^*,\mu_*)$ and $\widetilde{W}^\mathrm{cs}(0,\mu_*)$ of dimension four and three, respectively, inside $\R^5$. We say that $\rpf_*$ is a \emph{regular patterned front} if it satisfies the following conditions:

\begin{definition}[Regular patterned front]\label{d:rpf}
An element $\rpf_*=(\mu_*,\wl_*,u_\mathrm{wt}^*,u_\mathrm{het}^*)\in\mathcal{X}$ is called a \emph{regular patterned front} if it is a patterned front and the manifolds $\widetilde{W}^\mathrm{cu}(u_\mathrm{wt}^*,\mu_*)$ and $\widetilde{W}^\mathrm{cs}(0,\mu_*)$ intersect transversely in $\R^5$ along the heteroclinic orbit $(u_\mathrm{het}^*(x),\mu_*)$.
\end{definition}

We know more about the manifolds $\widetilde{W}^\mathrm{cu}(u_\mathrm{wt}^*,\mu_*)$ and $\widetilde{W}^\mathrm{cs}(0,\mu_*)$. Firstly, we know that
\[
\widetilde{W}^\mathrm{cs}(0,\mu_*) = \bigcup_{\mu\;\mathrm{near}\;\mu_*} (W^\mathrm{s}(0,\mu),\mu)
\]
is given by the union of the stable manifolds of the equilibrium $u=0$ for (\ref{e:ode}) as the parameter $\mu$ is varied (with the $\mu$-component appended). Next, Definition~\ref{d:pf}(ii) states that $u_\mathrm{wt}^*(\wl_*\cdot)$ is a symmetric hyperbolic periodic orbit of (\ref{e:ode}) for $\mu=\mu_*$. It is not difficult to prove that $u_\mathrm{wt}^*(\wl_*\cdot)$ is therefore also elementary as defined in \cite[Definition~1.5]{D} (see Lemma~\ref{l:1} in \S\ref{s:3}). Hence, by \cite[Proposition~1.6]{D}, it lies on a smooth two-parameter family of symmetric periodic orbits $u_\mathrm{per}(x;\kappa,\mu)$ of (\ref{e:ode}) with minimal period $\wl(\kappa,\mu)$ near $\wl_*$, which is parametrized by an internal parameter $\kappa$ and the system parameter $\mu$ for all $\mu$ near $\mu_*$ so that $u_\mathrm{wt}^*(\wl_*x)=u_\mathrm{per}(x;0,\mu_*)$, and $\partial_\kappa u_\mathrm{per}(x;0,\mu_*)$ does not vanish. We then know that
\[
\widetilde{W}^\mathrm{cu}(u_\mathrm{wt}^*,\mu_*) = \bigcup_{\mu\;\mathrm{near}\;\mu_*} \left(W^\mathrm{cu}(u_\mathrm{per}(\cdot;0,\mu);\mu),\mu\right)
\]
is determined by the union of the center-unstable manifolds of the $\mu$-dependent periodic orbits $u_\mathrm{per}(\cdot;0,\mu)$ for (\ref{e:ode}) as $\mu$ varies, where the underlying center manifolds coincide for each fixed $\mu$ with the smooth two-dimensional manifold $\{u_\mathrm{per}(x;\kappa,\mu)\colon x\in\R,\,\kappa\mbox{ near }0\}$.

If $\rpf_*=(\mu_*,\wl_*,u_\mathrm{wt}^*,u_\mathrm{het}^*)\in\mathcal{X}$ is a regular patterned front, the tangent spaces of $\widetilde{W}^\mathrm{cu}(u_\mathrm{wt}^*,\mu_*)$ and $\widetilde{W}^\mathrm{cs}(0,\mu_*)$ computed at $u_\mathrm{het}^*(0)$ intersect in a two-dimensional space. In particular, the intersection of the manifolds $\widetilde{W}^\mathrm{cu}(u_\mathrm{wt}^*,\mu_*)$ and $\widetilde{W}^\mathrm{cs}(0,\mu_*)$ forms a two-dimensional manifold near $u_\mathrm{het}^*(x)$. Taken together with the characterization of these manifolds given above, this demonstrates that, after factoring out the time-shift invariance of these manifolds, regular patterned fronts $\rpf_*\in\mathcal{X}$ lie on locally unique, smooth one-parameter branches. Thus, given a regular patterned front $\rpf_*\in\mathcal{X}$, there is an $\eta>0$ and a smooth function $\rpf:(-\eta,\eta)\to\mathcal{X}$ with $\rpf(0)=\rpf_*$ so that $\rpf(s)$ is a regular patterned front for each $s\in(-\eta,\eta)$ and there are no other regular patterned states near $\rpf_*$ in $\mathcal{X}$ (except for translating its $u_\mathrm{het}$ component in the variable $x$). We assume that (\ref{e:ode}) admits a closed loop consisting of regular patterned states as encoded in the next hypothesis:

\begin{hypothesis}\label{h:3}
There is a function
\[
\mathcal{S} \colon \quad
\R/\Z \longrightarrow \mathcal{X}, \quad
s \longmapsto (\mu(s), \wl(s), u_\mathrm{wt}(\cdot;s), u_\mathrm{het}(\cdot;s))
\]
of class $C^1$ so that $\mathcal{S}(s)$ is a regular patterned front for each $s$ with $\mathcal{S}_s(s)\neq0$ for all $s$.
\end{hypothesis}

While we will not use this interpretation, we can think of $s$ as the arclength parametrization of a closed curve of regular patterned fronts.

We now return to the asymptotic phase $\psi_*\in S^1$ that we introduced in Definition~\ref{d:pf}(iii) for a patterned front and investigate how it changes along the loop of regular patterned fronts described by $\mathcal{S}$. The three-dimensional center-unstable manifolds $W^\mathrm{cu}(u_\mathrm{wt}(\wl(s)\cdot;s);\mu(s))$ inherit a Riemannian metric from the ambient space $\R^4$. For each $\delta>0$ sufficiently small, we use the resulting geodesic distance to define the two-dimensional manifolds $W^\mathrm{cu}_\delta(u_\mathrm{wt}(\wl(s)\cdot;s);\mu(s))$ as the set of all points in $W^\mathrm{cu}(u_\mathrm{wt}(\wl(s)\cdot;s);\mu(s))$ that have distance $\delta$ from the two-dimensional center manifold $W^\mathrm{c}(u_\mathrm{wt}(\wl(s)\cdot;s);\mu(s))$. We then redefine $u_\mathrm{het}(\cdot;s)$ so that $u_\mathrm{het}(0;s)\in W^\mathrm{cu}_\delta(u_\mathrm{wt}(\wl(s)\cdot;s);\mu(s))$ for all $s\in\R/\Z$; we note that this is simply a normalization of the time-shift of the heteroclinic orbit and that the resulting function $u_\mathrm{het}(\cdot;s)$ will still be $C^1$ in $s$. Starting with $s=0$, we then define the $C^1$-function $\psi:[0,1]\to S^1$ so that $\psi(s)$ is the unique phase of the base point of the strong unstable fiber of $u_\mathrm{wt}(\wl(s)\cdot;s)$ that contains $u_\mathrm{het}(0;s)$. The phase function $\psi$ is then automatically periodic in $s$ so that $\psi\in C^1(\R/\Z,S^1)$. As we shall see next, the phase function is important since it characterizes the nature of the branches of localized patterned states that are constructed by gluing regular patterned fronts and their symmetric counterparts together.

\paragraph{Path lifting:} Let $\pi\colon\R\to S^1, x\mapsto\rme^{\rmi x}$ be the universal covering projection of $S^1$. Given a path $\gamma\in C^0([0,1],S^1)$, we say that a path $\widetilde{\gamma}\in C^0([0,1],\R)$ is a lifting of $\gamma$ if $\gamma=\pi\circ\widetilde{\gamma}$. Given an integer $n$, there exists a unique lifting $\tilde{\gamma}$ of $\gamma$ with $\widetilde{\gamma}(0)=\gamma(0)+2\pi n$.

We now apply these concepts to the phase function $\psi$ considered as a path $\psi\in C^0([0,1],S^1)$, that is, after initially ignoring that $\psi(s)$ is periodic in its argument $s$. Let $\widetilde{\psi}\in C^0([0,1],\R)$ be the unique lifting of $\psi\in C^0([0,1],S^1)$ that satisfies $\widetilde{\psi}(0)=\psi(0)$. There are then two distinct cases:
\begin{compactitem}
\item \textbf{Isolas:} If $\psi$ is homotopic to a constant function in $C^0(\R/\Z,S^1)$, then necessarily $\widetilde{\psi}(0)=\widetilde{\psi}(1)$ and therefore $\widetilde{\psi}\in C^1(\R/\Z,\R)$.
\item \textbf{Snaking:} If $\psi$ is not homotopic to a constant function in $C^0(\R/\Z,S^1)$, then necessarily $\widetilde{\psi}(0)\neq\widetilde{\psi}(1)$. In this case, we can extend $\widetilde{\psi}\in C^0([0,1],\R)$ uniquely to a function $\bar{\psi}\in C^0(\R^+,\R)$ via $\bar{\psi}(s+n):=(\widetilde{\psi}(1)-\widetilde{\psi}(0))n+\widetilde{\psi}(s)$ for $n\in\N$ and $0\leq s<1$, and the function $\bar{\psi}$ is unbounded and in $C^1$. Possibly after reversing $s\mapsto-s$, we can assume that $\bar{\psi}(s)\to\infty$ as $s\to\infty$.
\end{compactitem}

Finally, for each $\epsilon>0$ and each $s\in S^1$, we define the $\epsilon$-neighborhood
\[
\mathcal{V}^\mathrm{het}_\epsilon(s) :=
U_\epsilon\left(\left\{u_\mathrm{het}(x;s)\colon x\in\R\right\}\right) \cup
U_\epsilon\left(\left\{\mathcal{R}u_\mathrm{het}(x;s)\colon x\in\R\right\}\right) 
\]
of the heteroclinic cycle consisting of the heteroclinic connection $u_\mathrm{het}(x;s)$ and its reversed counterpart as well as the $\epsilon$-neighborhood
\[
\mathcal{V}^\mathrm{wt}_\epsilon(s) := U_\epsilon\left(\left\{u_\mathrm{wt}(\wl(s)x;s)\colon x\in\R\right\}\right)
\]
of the wave train $u_\mathrm{wt}(\wl(s)x;s)$. We can now define what we mean by a localized patterned state.

\begin{definition}[Symmetric localized patterned states]\label{d:lps}
Given $\epsilon>0$ and $s\in\R/\Z$, we say that $u_\mathrm{hom}(x)$ is a localized patterned state if $u_\mathrm{hom}(x)$ is a solution of (\ref{e:ode}) for some $\mu$ near $\mu(s)$, we have $u_\mathrm{hom}(x)\in W^\mathrm{s}(0;\mu)\cap W^\mathrm{u}(0;\mu)$, the entire orbit $u_\mathrm{hom}(\cdot)$ lies in $\mathcal{V}^\mathrm{het}_\epsilon(s)$, and there is an $L\gg1$ so that $u_\mathrm{hom}(x)$ spends time $2L$ in the neighborhood $\mathcal{V}^\mathrm{wt}_\epsilon(s)$. For a given $\varphi_0\in\{0,\pi\}$, we say that $u_\mathrm{hom}$ is $\varphi_0$-symmetric if the orbit $u_\mathrm{hom}(\cdot)$  intersects $\Fix\mathcal{R}$ near $u_\mathrm{wt}(\wl(s)\varphi_0;s)$.
\end{definition}

We note that \citep[Lemma~3]{VF} shows that the last condition in Definition~\ref{d:lps} determines $\varphi_0\in\{0,\pi\}$ uniquely. We have the following theorem.

\begin{theorem}\label{t:1}
Assume that Hypotheses~\ref{h:1}-\ref{h:3} are met, then there are constants $\ell_*\gg1$, $\beta>0$, and $\epsilon>0$ so that the following is true separately for each $\varphi_0\in\{0,\pi\}$:
\begin{compactitem}
\item \textbf{Isolas:} If $\psi$ is homotopic to a constant function in $C^0(\R/\Z,S^1)$, then (\ref{e:ode}) has a $\varphi_0$-symmetric localized patterned state if and only if there are $s\in\R/\Z$ and $n\in\N$ with $n\geq\ell_*$ so that 
\[
(L,\mu) = (\widetilde{\psi}(s)+2\pi n-\varphi_0,\mu(s)) + \rmO(\rme^{-\beta n}).
\]
For each fixed $n$, the branch parametrized by $s\in\R/\Z$ forms a closed smooth loop.
\item \textbf{Snaking:} If $\psi$ is not homotopic to a constant function in $C^0(\R/\Z,S^1)$, then (\ref{e:ode}) has a $\varphi_0$-symmetric localized patterned state if and only if there is an $s\in\R$ with $s\geq\ell_*$ so that 
\[
(L,\mu) = (\bar{\psi}(s)-\varphi_0,\mu(s)) + \rmO(\rme^{-\beta s}),
\]
and the resulting branch is smooth and unbounded.
\end{compactitem}
\end{theorem}

%%%%%%%%%%%%%%%%%%%%%%%%%%%%%%%%%%%%%%%%%%%%%%%%%%%%%%%%%%%%%%%%%%%%%%%%%%%%

\section{Outline of the proof of Theorem~\ref{t:1}}\label{s:2}

Throughout, we assume that Hypotheses~\ref{h:1}-\ref{h:3} are met.

First, for each $\mu$ near $\mu(s)$, we can smoothly straighten out the two-dimensional center manifold of the wave trains $u_\mathrm{wt}(\wl(s)x;s)$, which is fibered by the two-parameter family of periodic orbits, and the corresponding strong stable and unstable invariant foliations. After this local coordinate transformation, these manifolds and foliations no longer depend on $\mu$. As outlined in \citep{Snaking}, the vector field (\ref{e:ode}) near the wave trains $u_\mathrm{wt}(\wl(s)x;s)$ for $\mu$ near $\mu(s)$ can then be transformed into the normal form
\begin{eqnarray}
v^\mathrm{c}_x & = & 1 + \tilde{h}^\kappa(v^\kappa,\mu) v^\kappa + h^\mathrm{c}(v,\mu) v^\mathrm{s} v^\mathrm{u} \nonumber \\ \label{e:nf}
v^\kappa_x & = & h^\kappa(v,\mu) v^\mathrm{s} v^\mathrm{u} \\ \nonumber
v^\mathrm{s}_x & = & -[\alpha(\mu) + h^\mathrm{s}_1(v,\mu) v^\mathrm{s} + h^\mathrm{s}_2(v,\mu) v^\mathrm{u}] v^\mathrm{s} \\ \nonumber
v^\mathrm{u}_x & = & [\alpha(\mu) + h^\mathrm{u}_1(v,\mu) v^\mathrm{s} + h^\mathrm{u}_2(v,\mu) v^\mathrm{u}] v^\mathrm{u},
\end{eqnarray}
where $v=(v^\mathrm{c},v^\kappa,v^\mathrm{s},v^\mathrm{u})\in\mathcal{V}^\mathrm{wt}_\delta:=S^1\times I\times I\times I$ for $I=[-\delta,\delta]$ with $\delta>0$ small enough. The reverser $\mathcal{R}$ acts on $v$ via
\[
\mathcal{R}(v^\mathrm{c},v^\kappa,v^\mathrm{s},v^\mathrm{u}) = (-v^\mathrm{c},v^\kappa,v^\mathrm{u},v^\mathrm{s}).
\]
In these coordinates, we define the sections
\begin{equation}\label{e:Sigma}
\Sigma_\mathrm{in} = S^1\times I\times\{v^\mathrm{s}=\delta\} \times I, \qquad
\Sigma_\mathrm{out} = S^1\times I\times I\times\{v^\mathrm{u}=\delta\}.
\end{equation}
We can solve (\ref{e:nf}) as in \citep[Lemma~3.1]{Snaking}: For each $L\gg1$, $\varphi_0\in\{0,\pi\}$, $v^\kappa_0:=v^\kappa(0)\in I$, and $\mu$ near $\mu(s)$, the system (\ref{e:ode}) has a unique reversible solution $v(x)$ that passes inside $\mathcal{V}^\mathrm{wt}_\delta$ from $\Sigma_\mathrm{in}$ to $\Sigma_\mathrm{out}$ in time $2L$, and we have the expansion
\begin{equation}\label{e:vl}
v(L) = \left(\varphi_0+(1+\rmO(v^\kappa_0))L+\rmO(\rme^{-\beta L}),v^\kappa_0+\rmO(\rme^{-\beta L}),\rmO(\rme^{-\beta L}),\delta\right).
\end{equation}
for some $\beta>0$. Furthermore, these solutions depend smoothly on $(\mu,L)$.

Next, we consider the intersection of $W^\mathrm{s}(0;\mu)$ with $\Sigma_\mathrm{out}$ for $\mu$ near $\mu(s)$. Since $u_\mathrm{het}(\cdot;s)$ is a regular patterned front, Definition~\ref{d:rpf} implies that for each fixed $s$ the manifold $W^\mathrm{s}(0;\mu(s))$ intersects the section $\Sigma_\mathrm{out}$ transversely in a one-dimensional manifold that contains $u_\mathrm{het}(0;s)$. Hence, for each fixed $\mu=\mu(s)+\nu$ with $|\nu|\ll1$, the manifold $W^\mathrm{s}(0;\mu)$ will also intersect $\Sigma_\mathrm{out}$ transversely in a curve. We denote by $w^\mathrm{s}(s)$ a nonzero basis vector in the tangent space of $W^\mathrm{s}(0;\mu(s))\cap\Sigma_\mathrm{out}$ at $u_\mathrm{het}(0;s)$ so that  $w^\mathrm{s}(s)$ depends smoothly on $s$. We can then parametrize the intersections $W^\mathrm{s}(0;\mu(s)+\nu)\cap\Sigma_\mathrm{out}$ by $(\nu,a)$, where $a\mapsto aw^\mathrm{s}(s)$ parametrizes the tangent space of the intersection at $\nu=0$. Using the definition of the asymptotic phase function $\psi(s)$, we conclude that
\begin{equation}\label{e:wsout}
W^\mathrm{s}(0;\mu(s)+\nu) \cap \Sigma_\mathrm{out} =
\left\{ v=\begin{pmatrix}v^\mathrm{c} \\ v^\kappa \\ v^\mathrm{s} \\ \delta \end{pmatrix}\in\Sigma_\mathrm{out}\colon
\begin{pmatrix} v^\mathrm{c} \\ v^\kappa \\ v^\mathrm{s} \end{pmatrix} =
\begin{pmatrix} \psi(s)+\rmO(|\nu|+|a|) \\ \rmO(|\nu|+|a|) \\ c_1(s)\nu+c_2(s)a+\rmO(\nu^2+a^2) \end{pmatrix},\;
a\in I \right\}
\end{equation}
for constants $c_1(s),c_2(s)\in\R$. Since $u_\mathrm{het}(\cdot;s)$ is a regular patterned front, and since the center manifold and the strong stable and unstable invariant foliations of the wave train no longer depend on $\mu$ due to our local coordinate transformation, Definition~\ref{d:rpf} implies that $c_1(s)^2+c_2(s)^2>0$ for all $s$. We remark that we use the three parameters $(s,\nu,a)$ to parametrize the two-dimensional manifold in (\ref{e:wsout}), and we can therefore always set $\nu=0$ or $a=0$.

We are now in a position to construct symmetric localized patterned states that spent time $2L$ in $\mathcal{V}^\mathrm{wt}_\delta$ by matching $v(L)$ from (\ref{e:vl}) with the manifold $W^\mathrm{s}(0;\mu(s)+\nu)\cap\Sigma_\mathrm{out}$ that we characterized in (\ref{e:wsout}). Using (\ref{e:vl}) and (\ref{e:wsout}), the matching condition $v(L)\in W^\mathrm{s}(0;\mu(s)+\nu)\cap\Sigma_\mathrm{out}$ becomes
\begin{equation}\label{e:m}
\begin{pmatrix} \varphi_0+(1+\rmO(v^\kappa_0))L+\rmO(\rme^{-\beta L}) \\ v^\kappa_0+\rmO(\rme^{-\beta L}) \\ \rmO(\rme^{-\beta L}), \end{pmatrix} =
\begin{pmatrix} \psi(s)+\rmO(|\nu|+|a|) \\ \rmO(|\nu|+|a|) \\ c_1(s)\nu+c_2(s)a+\rmO(\nu^2+a^2) \end{pmatrix}
\end{equation}
in $S^1\times I\times I$, where $L\gg1$ and $(v^\kappa_0,a,\nu)$ are close to zero.

\paragraph{Isolas:} Assume first that $\psi$ is homotopic to a constant function in $C^0(\R/\Z,S^1)$, so that the lifting $\widetilde{\psi}:[0,1]\to\R$ is periodic in $s$. In this case, we can lift the first equation in (\ref{e:m}) to $\R$. Given any natural number $n\in\N$, we obtain
\[
\begin{pmatrix} \varphi_0+(1+\rmO(v^\kappa_0))L+\rmO(\rme^{-\beta L}) \\ v^\kappa_0+\rmO(\rme^{-\beta L}) \\ \rmO(\rme^{-\beta L}), \end{pmatrix} =
\begin{pmatrix} 2\pi n+\widetilde{\psi}(s)+\rmO(|\nu|+|a|) \\ \rmO(|\nu|+|a|) \\ c_1(s)\nu+c_2(s)a+\rmO(\nu^2+a^2) \end{pmatrix}
\]
in $\R\times I\times I$, where the right-hand side is periodic in $s$, and where we recall that $c_1(s)^2+c_2(s)^2>0$. For each sufficiently large $n\gg1$, we can solve this system uniquely for $(L,v^\kappa_0,\nu)$ with $a=0$ as a function of $s$ when $c_1(s)\neq0$ and for $(L,v^\kappa_0,a)$ with $\nu=0$ as a function of $s$ when $c_1(s)=0$. We note that the first case is the generic case (since we are then away from saddle-node bifurcations of the heteroclinic connections $u_\mathrm{het}(\cdot,s)$). In either case, uniformly in $n\gg1$, we have the expansion
\[
(L,v^\kappa,\mu) = (\widetilde{\psi}(s) + 2\pi n - \varphi_0,0,\mu(s)) + \rmO(\rme^{-\beta n}), \qquad s\in\R/\Z,
\]
and the resulting branches lie for each fixed $n$ on a closed loop.

\paragraph{Snaking:} If $\psi$ is not homotopic to a constant function in $C^0(\R/\Z,S^1)$, we can extend the lifting uniquely to a continuous unbounded function $\bar{\psi}:\R^+\to\R^+$. Lifting the first equation in (\ref{e:m}) to $\R$, we obtain
\[
\begin{pmatrix} \varphi_0+(1+\rmO(v^\kappa_0))L+\rmO(\rme^{-\beta L}) \\ v^\kappa_0+\rmO(\rme^{-\beta L}) \\ \rmO(\rme^{-\beta L}), \end{pmatrix} =
\begin{pmatrix} \bar{\psi}(s)+\rmO(|\nu|+|a|) \\ \rmO(|\nu|+|a|) \\ c_1(s)\nu+c_2(s)a+\rmO(\nu^2+a^2) \end{pmatrix}
\]
in $\R\times I\times I$, where the right-hand side now depends on $s\in\R^+$. As above, for each sufficiently large $s\gg1$, we can solve this system uniquely for $(L,v^\kappa_0,\nu)$ with $a=0$ as a function of $s$ when $c_1(s)\neq0$ and for $(L,v^\kappa_0,a)$ with $\nu=0$ as a function of $s$ when $c_1(s)=0$. In both cases, we have the expansion
\[
(L,v^\kappa,\mu) = (\bar{\psi}(s)-\varphi_0,0,\mu(s)) + \rmO(\rme^{-\beta s}), \qquad s\in\R
\]
for $s\gg1$, and the resulting branch is unique and unbounded. This  completes the proof of Theorem~\ref{t:1}.

%%%%%%%%%%%%%%%%%%%%%%%%%%%%%%%%%%%%%%%%%%%%%%%%%%%%%%%%%%%%%%%%%%%%%%%%%%%%

\section{Reversible hyperbolic periodic orbits are elementary}\label{s:3}

Consider the system
\begin{equation}\label{a:ode}
u_x = f(u), \qquad u\in\R^4
\end{equation}
in $\R^4$, where $f$ is $C^\infty$. We assume that this system admits a linear reverser $\mathcal{R}:\R^4\to\R^4$ so that $\mathcal{R}^2=1$, $\dim\Fix\mathcal{R}=2$, and $f(\mathcal{R}u)=-\mathcal{R}f(u)$ for all $u$. After an appropriate coordinate transformation, we may also assume that $\Fix(\mathcal{R})$ and $\Fix(-\mathcal{R})$ are orthogonal to each other. Assume that $u_*(x)$ is a symmetric periodic orbit with Floquet multipliers $\rme^{\pm\alpha}$ for some $\alpha>0$ and a Floquet multiplier at $1$ (necessarily with algebraic multiplicity two). The following lemma should be well known, but I could not locate a reference.

\begin{lemma}\label{l:1}
The periodic orbit $u_*(x)$ is part of a smooth one-parameter family of periodic orbits $u(x;\kappa)$ with $u(0;\kappa)\in\Fix\mathcal{R}$ for $\kappa$ near zero so that $u_*(0)=u(0;0)$ and $\partial_\kappa u(\cdot;0)\neq0$.
\end{lemma}

\begin{proof}
Let $2T>0$ be the period of $u_*(x)$ and denote the flow of (\ref{a:ode}) by $\Phi_t$. We consider sections $\Sigma_p$ and $\Sigma_q$ at $p=u_*(0)$ and $q=u_*(T)$ so that $f(p)$ and $f(q)$ are perpendicular to these sections at $u=p$ and $u=q$, respectively. Let $\Psi:\Sigma_p\to\Sigma_q$ denote the first-return map between these sections. By \cite[Proposition~1.6]{D}, it suffices to show that $\rmD\Psi(p)(\Fix\mathcal{R})$ is transverse to $\Fix\mathcal{R}$ in $\Sigma_q$.

Assume that this transversality condition is not met, then necessarily $\rmD\Psi(p)(\Fix\mathcal{R})=\Fix\mathcal{R}$. Floquet theory implies that $\rmD\Phi_t(p)=C(t)\rme^{Bt}$ where $C(2T)=C(0)$ is the identity. Let $v^\mathrm{s,u}_0$ denote the stable and unstable eigenvectors of $\rmD\Phi_{2T}(p)$ and define $v^\mathrm{s,u}_T:=C(T)v^\mathrm{s,u}_0$. Reversibility shows that we can pick $v^\mathrm{u}_0:=\mathcal{R}v^\mathrm{s}_0$ so that $v^\mathrm{s}_0+v^\mathrm{u}_0\in\Fix\mathcal{R}$ and, furthermore, $v^\mathrm{u}_T=\mathcal{R}v^\mathrm{s}_T$. Our assumption on $\rmD\Psi(p)$ implies that $\rmD\Phi_{T}(p)(v^\mathrm{s}_0+v^\mathrm{u}_0)\in\Fix\mathcal{R}$. Since $\rmD\Phi_{T}(p)(v^\mathrm{s}_0+v^\mathrm{u}_0)=\rme^{-\alpha/2}v^\mathrm{s}_T+\rme^{\alpha/2}v^\mathrm{u}_T$, we conclude that $\rme^{-\alpha}v^\mathrm{s}_T=v^\mathrm{s}_T$, which is possible only if $\alpha=0$. This contradicts our assumption and therefore completes the proof of the lemma.
\end{proof}

%%%%%%%%%%%%%%%%%%%%%%%%%%%%%%%%%%%%%%%%%%%%%%%%%%%%%%%%%%%%%%%%%%%%%%%%%%%%

\bibliographystyle{sandstede}
\bibliography{Snaking_Continuation}

%%%%%%%%%%%%%%%%%%%%%%%%%%%%%%%%%%%%%%%%%%%%%%%%%%%%%%%%%%%%%%%%%%%%%%%%%%%%

\end{document}